\documentclass{article}
\usepackage {amssymb}
\usepackage {amsmath}
\usepackage {amsmath}
\usepackage {amsthm}
\usepackage{graphicx}
\usepackage {amscd}
\usepackage[colorlinks,linkcolor=black, anchorcolor=black, citecolor=red]{hyperref}
\newtheorem{thm}{Theorem}
\newtheorem{prop}{Proposition}
\newtheorem{rem}{Remark}

\newtheorem{exmp}{Example}
\newtheorem{lem}{Lemma}

\begin{document}
\title{On the limit behavior of metrics in continuity method to K\"{a}hler-Einstein problem in toric Fano case}
\author{Chi Li}
\date{}
\maketitle

\begin{abstract}
\noindent ABSTRACT: This is a continuation of paper \cite{Li}. On
any toric Fano manifold, we discuss the behavior of limit metric of
a sequence of metrics, which are solutions to a continuity family of
complex Monge-Amp\`{e}re equations in K\"{a}hler-Einstein problem.
We show that the limit metric satisfies a singular complex
Monge-Amp\`{e}re equation. This shows the conic type singularity for
the limit metric. The information of conic type singularities can be
read from the geometry of the moment polytope.
\end{abstract}

\section{Introduction}
Let $(X,J)$ be a Fano manifold, that is, $K_X^{-1}$ is ample. Fix a reference K\"{a}hler metric $\omega\in c_1(X)$. Its Ricci curvature
$Ric(\omega)$ also lies in $c_1(X)$. So there exists $h_\omega\in C^{\infty}(X)$ such that
\[
Ric(\omega)-\omega=\partial\bar{\partial}h_\omega,\quad \int_X e^{h_\omega}\omega^n=\int_X\omega^n
\]
Consider the following family of Monge-Amp\`{e}re equations.
\begin{align}
(\omega+\partial\bar{\partial}\phi_t)^n=e^{h_\omega-t\phi}\omega^n\tag*{$(*)_t$}\label{CMAt}
\end{align}
Let $R(X)=\sup \{t: \mbox{\ref{CMAt} is solvable }\}$. Then Sz\'{e}kelyhidi proved that
\begin{prop}[\cite{S}]
\[
R(X)=\sup\{t: \exists \mbox{ a K\"{a}hler metric }\; \omega\in
c_1(X) \mbox{ such that } Ric(\omega)> t\omega\}
\]
\end{prop}
In particular, $R(X)$ is independent of reference metric $\omega$. In \cite{Li}, we determined $R(X)$ for any toric Fano manifold.

A toric Fano
manifold $X_\triangle$ is determined by a reflexive lattice polytope
$\triangle$ (For details on toric manifolds, see \cite{Oda}). For
example, let $Bl_p\mathbb{P}^2$ denote the manifold obtained by blowing up one point on $\mathbb{P}^2$. Then $Bl_p\mathbb{P}^2$ is a toric Fano manifold and is determined by the following polytope.

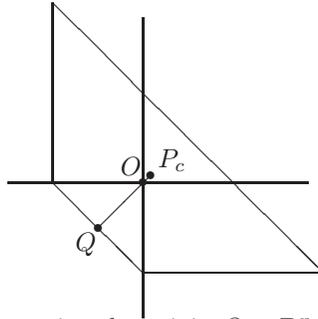
\begin{figure}[h]
 \begin{center}\label{figure1}
\setlength{\unitlength}{1mm}
\begin{picture}(30,30)
\put(-18, 0){\line(1,0){40}} \put(0, -18){\line(0,1){40}}
\put(-12,0){\line(0,1){24}} \put(-12,0){\line(1,-1){12}}
\put(-12,24){\line(1,-1){36}} \put(0,-12){\line(1,0){24}}
\put(-6,-6){\line(1,1){7}} \put(0,0){\circle*{1}}
\put(1,1){\circle*{1}} \put(-6,-6){\circle*{1}} \put(-9,-9){$Q$}
\put(-3,1){$O$} \put(2,2){$P_c$}
\end{picture}
\end{center}
\end{figure}

\vspace*{10mm}

Any such polytope $\triangle$ contains the origin
$O\in\mathbb{R}^n$. We denote the barycenter of $\triangle$ by
$P_c$. If $P_c\neq O$, the ray
$P_c+\mathbb{R}_{\ge0}\cdot\overrightarrow{P_c O}$ intersects the
boundary $\partial\triangle$ at point $Q$.

\begin{thm}\cite{Li}\label{thm1}
If $P_c\neq O$,
\[
R(X_\triangle)=\frac{|\overline{OQ}|}{|\overline{P_cQ}|}
\]
Here $|\overline{OQ}|$, $|\overline{P_cQ}|$ are lengths of line
segments $\overline{OQ}$ and $\overline{P_cQ}$. If $P_c=O$, then
there is K\"{a}hler-Einstein metric on $X_\triangle$ and
$R(X_\triangle)=1$.
\end{thm}

The next natural problem is what the limit metric looks like as
$t\rightarrow R(X)$. For the special example $X=Bl_p\mathbb{P}^2$,
which is also the projective compactification of total space of line
bundle $\mathcal{O}(-1)\rightarrow \mathbb{P}^2$.  Sz\'{e}kelyhidi
\cite{S} constructed a sequence of K\"{a}hler metric $\omega_t$,
with $Ric(\omega_t)\ge t\omega_t$ and $\omega_t$ which converge to a
metric with conic singularty along the divisor $D_\infty$ of conic
angle $2\pi\times 5/7$, where $D_\infty$ is divisor at infinity
added in projective compactification. Shi-Zhu \cite{ShZh} proved
that rotationally symmetric solutions to the continuity equations
\ref{CMAt} converge to a metric with conic singularity of conic
angle $2\pi\times 5/7$ in Gromov-Hausdorff sense, which seems to be
the first strict result on behavior of solutions to \ref{CMAt}. Note
that by the theory of Cheeger-Colding-Tian \cite{CCT}, the limit
metric in Gromov-Hausdorff sense should have complex codimension 1
conic type singularities if we only have the positive lower Ricci
bounds.

For the more general toric case, if we use a special toric metric,
which is just the Fubini-Study metric in the projective embedding
given by the vertices of the polytope, then, after transforming by
some biholomorphic automorphism, we prove there is a sequence of
K\"{a}hler metrics which solve the equation \ref{CMAt}, and converge
to a limit metric satisfying a singular complex Monge-Amp\`{e}re
equation (Also see equivalent real version in Theorem \ref{thm3}). This generalizes the result of \cite{ShZh} for
the special reference Fubini-Study metric. 

Precisely, let $\{p_\alpha; \alpha=1, \dots, N\}$ be all the vertex lattice points of $\triangle$ and $\{s_\alpha; \alpha=1, \dots, N\}$ be the corresponding holomorphic sections of $K_{X_\triangle}^{-1}$.  Then we take  reference metric to be
\[
\omega=\omega_{FS}=\partial\bar{\partial}\log\sum_{\alpha=1}^N|s_\alpha|^2
\]
which is the pull-back of the Fubini-Study metric of $\mathbb{CP}^{N-1}$ under Kodaira embedding by $\{s_\alpha\}$. 
Now using the same notation as that in Theorem \ref{thm1}, let $\mathcal{F}$ be the minimal face of $\triangle$ containing 
$Q$. Let $\{p_k^{\mathcal{F}}\}$ be the vertex lattice points of $\mathcal{F}$, then they correspond to a sub-linear system
$\mathfrak{L}_\mathcal{F}$ of $|-K^{-1}_{X_\triangle}|$. We let $Bs(\mathfrak{L}_\mathcal{F})$ denote the base locus of this sub-linear system. Also let $\sum_{\alpha}{}'$ denote the sum $\sum_{p_k^{\mathcal{F}}}$, then we have

\begin{thm}\label{thm2}
After biholomorphic transformation $\sigma_t: X_\triangle\rightarrow X_\triangle$, there is a subsequence $t_i\rightarrow R(X)$, such that $\sigma_{t_i}^*\omega_{t_i}$
converge to a K\"{a}hler current $\omega_\infty=\omega+\partial\bar{\partial}\psi_\infty$, with $\psi_\infty\in L^\infty(X_\triangle)\cap
C^{\infty}(X_\triangle\backslash Bs(\mathfrak{L}_{\mathcal{F}}))$,  which satisfies a complex Monge-Amp\`{e}re equation of the form
\begin{equation}\label{singMA1}
(\omega+\partial\bar{\partial}\psi_\infty)^n=e^{-R(X)\psi_\infty}\left(\sum_\alpha{}'b_\alpha\|s_\alpha\|^2\right)^{-(1-R(X))}\Omega
\end{equation}
Here $\Omega=e^{h_\omega}\omega^n$ is a smooth volume form. For each vertex lattice point $p_\alpha^{\mathcal{F}}$ of $\mathcal{F}$, $b_\alpha$ is a constant satisfying $0<b_\alpha<1$.  $\|\cdot\|=\|\cdot\|_{FS}$ is (up to multiplication of a constant) the Fubini-Study metric on 
$K_{X_{\triangle}}^{-1}$. 
In particular
\begin{equation}\label{singric}
Ric(\omega_{\psi_\infty})=R(X)\omega_{\psi_\infty}+(1-R(X))\partial\bar{\partial}\log(\sum_\alpha{}' b_\alpha |s_\alpha|^2)
\end{equation}
\end{thm}

From this theorem we can expect the conic behavior at generic point of the singularities of the limit
metric, and we can read out the place of singularities and the
conic angles from the geometry of the polytope.
See Section \ref{conicint} for discussions. 
In particular, this can give a
toric explanation of the special case $Bl_p\mathbb{P}^2$ just
mentioned (See example \ref{example1}).

Note that, although we can prove the limit metric is smooth outside the singular locus, to actually prove 
it's a conic metric along codimension one strata of singular set, we need
to prove more delicate estimate that we wish to discuss in future.  There are also difficulties for studying the behavior of limit metric around higher codimensional strata (See Remark \ref{overlook} and Example \ref{example2}). 

\textbf{Acknowledgement}: The author thanks Professor Gang Tian for
helpful discussions and constant encouragement. In particular, he
notified the author the Harnack estimates from \cite{T4}. The author is grateful to Professor Jian Song 
for carefully reading the first version of this paper, and pointing out a wrong statement of the theorem \ref{thm2} and an inacurracy in the proof of Harnack inequality.  
See Remark
\ref{overlook} and Remark \ref{gap} on these issues. The author thanks Professor Sormani for discussions on the Gromov-Hausdorff convergence. The author
also thanks Yalong Shi for informing him the result of \cite{ShZh}
which motivates this paper. 

\section{Consequence of estimates of Wang-Zhu}
The proof of Theorem \ref{thm1} is based on the methods of Wang-Zhu \cite{WZ}.

For a reflexive
lattice polytope $\triangle$ in
$\mathbb{R}^n=\Lambda\otimes_\mathbb{Z}\mathbb{R}$, we have a Fano
toric manifold $(\mathbb{C}^*)^n\subset X_\triangle$ with a
$(\mathbb{C}^*)^n$ action. In the following, we will sometimes just write $X$ for $X_\triangle$ for simplicity.

Let $(S^1)^n\subset(\mathbb{C}^*)^n$ be the standard real maximal
torus. Let $\{z_i\}$ be the standard coordinates of the dense orbit
$(\mathbb{C}^*)^n$, and $x_i=\log|z_i|^2$. We have
\begin{lem}
Any $(S^1)^n$ invariant K\"{a}hler metric $\omega$ on $X$ has a
potential $u=u(x)$ on $(\mathbb{C}^*)^n$, i.e.
$\omega=\frac{\sqrt{-1}}{2\pi}\partial\bar{\partial}u$. $u$ is a
proper convex function on $\mathbb{R}^n$, and satisfies the momentum
map condition:
\[
Du(\mathbb{R}^n)=\triangle
\]
Also,
\begin{equation}\label{toricvol}
\frac{(\partial\bar{\partial}u)^n/n!}{\frac{dz_1}{z_1}\wedge\frac{d\bar{z}_1}{\bar{z}_1}\cdots\wedge\frac{dz_n}{z_n}\wedge\frac{d\bar{z}_n}{\bar{z}_n}}=\det\left(\frac{\partial^2u}{\partial x_i\partial x_j}\right)
\end{equation}
\end{lem}

Let $\{p_\alpha;\; \alpha=1,\cdots, N\}$ be all the \textbf{vertex
lattice points} of $\triangle$. Each $p_\alpha$ corresponds to a
holomorphic section $s_\alpha\in H^0(X_\triangle,
K^{-1}_{X_\triangle})$. We can embed $X_\triangle$ into
$\mathbb{P}^{N}$ using $\{s_\alpha\}$. Define $\tilde{u}_0$ to be
the potential on $(\mathbb{C}^*)^n$ for the pull back of
Fubini-Study metric (i.e.
$\frac{\sqrt{-1}}{2\pi}\partial\bar{\partial}\tilde{u}_0=\omega_{FS}$):
\begin{equation}\label{u0}
\tilde{u}_0=\log\left(\sum_{\alpha=1}^N e^{<p_\alpha,x>}\right)+C
\end{equation}
$C$ is some constant determined by normalization condition:
\begin{equation}\label{normalize}
\int_{\mathbb{R}^n}e^{-\tilde{u}_0}dx=Vol(\triangle)=\frac{1}{n!}\int_{X_\triangle}\omega^n=\frac{c_1(X_\triangle)^n}{n!}
\end{equation}
By the normalization of $\tilde{u}_0$, it's easy to see that
\begin{equation}\label{torich}
\frac{e^{h_\omega}\omega^n}{\frac{dz_1}{z_1}\wedge\frac{d\bar{z}_1}{\bar{z}_1}\cdots\wedge\frac{dz_n}{z_n}\wedge\frac{d\bar{z}_n}{\bar{z}_n}}=e^{-\tilde{u}_0}
\end{equation}
\begin{rem}
 We only use vertex lattice points because, roughly speaking, later in Lemma \ref{limref}, vertex lattice points alone helps us to determine which sections become degenerate when doing biholomorphic transformation and taking limit.
 See remark \ref{vertex2}. We expect results similar to Theorem \ref{thm2} hold for
 general toric reference K\"{a}hler metric.
\end{rem}

So divide both sides of \ref{CMAt} by meromorphic volume form
$\frac{dz_1}{z_1}\wedge\frac{d\bar{z}_1}{\bar{z}_1}\cdots\wedge\frac{dz_n}{z_n}\wedge\frac{d\bar{z}_n}{\bar{z}_n}$,
We can rewrite the equations \ref{CMAt} as a family of real
Monge-Amp\`{e}re equations on $\mathbb{R}^n$:
\begin{align}
\det(u_{ij})=e^{-(1-t)\tilde{u}_0-tu}\tag*{$(**)_t$}\label{RMAt}
\end{align}
where $u$ is the potential for $\omega+\partial\bar{\partial}\phi$ on $(\mathbb{C}^*)^n$, and is related to $\phi$ in \ref{CMAt} by
\[
\phi=u-\tilde{u}_0
\]
For simplicity, let
\[
w_t(x)=tu(x)+(1-t)\tilde{u}_0
\]
Then $w_t$ is also a proper convex function on $\mathbb{R}^n$ satisfying $Dw_t(\mathbb{R}^n)=\triangle$. So it has a unique absolute minimum at point $x_t\in\mathbb{R}^n$. Let
\[
m_t=inf\{w_t(x):x\in\mathbb{R}^n\}=w_t(x_t)
\]

Then the main estimate of Wang-Zhu \cite{WZ} is that
\begin{prop}[\cite{WZ},See also \cite{D}]\label{wzest}
\begin{enumerate}
\item
there exists a constant $C$, independent of $t<R(X_\triangle)$, such
that
\[|m_t|<C\]
\item There exists $\kappa>0$ and a constant $C$, both independent of
$t<R(X_\triangle)$, such that
\begin{equation}\label{esti}
w_t\ge\kappa|x-x_t|-C
\end{equation}
\end{enumerate}
\end{prop}

\begin{prop}[\cite{WZ}]
the uniform bound of $|x_t|$ for any $0\le t\le t_0$, is equivalent
to that we can solve \ref{RMAt}, or equivalently solve \ref{CMAt},
for $t$ up to $t_0$. More precisely, (by the discussion in
introduction,) this condition is equivalent to the uniform
$C^0$-estimates for the solution $\phi_t$ in \ref{CMAt} for $t\in
[0,t_0]$.
\end{prop}

By the above proposition, we have
\begin{lem}
If $R(X_\triangle)<1$, then there exists a subsequence $\{x_{t_i}\}$
of $\{x_t\}$, such that
\[\lim_{t_i\rightarrow R(X_\triangle)}|x_{t_i}|=+\infty\]
\end{lem}

By the properness of $\tilde{u}_0$ and compactness of $\triangle$, we ge imediately that

\begin{lem}\label{observe1}
If $R(X_\triangle)<1$, then there exists a subsequence of
$\{x_{t_i}\}$ which we still denote by $\{x_{t_i}\}$, and
$y_\infty\in
\partial\triangle$, such that
\begin{equation}\label{cvyinf}
\lim_{t_i\rightarrow R(X_\triangle)}D\tilde{u}_0(x_{t_i})=y_\infty
\end{equation}
\end{lem}

To determine $R(X_\triangle)$ we use the key identity:

\begin{equation}\label{keyiden1}
\frac{1}{Vol(\triangle)}\int_{\mathbb{R}^n}D\tilde{u}_0e^{-w}dx=-\frac{t}{1-t}P_c
\end{equation}

\begin{rem}
This identity is a toric form of a general formula for solutions of equations \ref{CMAt}:
\[
-\frac{1}{V}\int_X div_\Omega(v)\omega_t^n=\frac{t}{1-t}F_{c_1(X)}(v)
\]
Here $\Omega=e^{h_\omega}\omega^n$. $v$ is any holomorphic vector field, and $div_\Omega(v)=\frac{\mathcal{L}_v\Omega}{\Omega}$ is the divergence
of $v$ with respect to $\Omega$.
\[
F_{c_1(X)}(v)=\frac{1}{V}\int_X v(h_\omega)\omega^n
\]
is the Futaki invariant in class $c_1(X)$ \cite{Fu}.
\end{rem}

By properness of $w_t$, the left handside of \eqref{keyiden1} is roughly $D\tilde{u}_0(x_t)$. As long as this is bounded away from
the boundary of the polytope, we can control the point $x_t$. So as $t$ goes to $R(X_\triangle)$, since $x_t$ goes to infinity
in $\mathbb{R}^n$, the left handside goes to a point on $\partial\triangle$, which is roughly $y_\infty$. To state a precise statement, assume
the reflexive polytope $\triangle$ is defined by
inequalities:
\begin{equation}\label{deftri}
\lambda_r(y)\ge-1,\;r=1,\cdots,K
\end{equation}
$\lambda_r(y)=\langle v_r,y\rangle$ are fixed linear functions. We
also identify the minimal face of $\triangle$ where $y_\infty$ lies:
\begin{eqnarray}\label{yinface}
&&\lambda_r(y_{\infty})=-1,\; r=1,\cdots,K_0\\
&&\lambda_r(y_{\infty})>-1,\; r=K_0+1,\cdots, K\nonumber
\end{eqnarray}

Then Theorem \ref{thm1} follows from
\begin{prop}\cite{Li}\label{main}
If $P_c\neq O$,
\[-\frac{R(X_\triangle)}{1-R(X_\triangle)}P_c\in\partial\triangle\]
Precisely,
\begin{equation}\label{result}
\lambda_r\left(-\frac{R(X_\triangle)}{1-R(X_\triangle)}P_c\right)\ge
-1
\end{equation}
Equality holds if and only if $r=1,\cdots,K_0$. So
$-\frac{R(X_\triangle)}{1-R(X_\triangle)}P_c$ and $y_\infty$ lie on
the same faces \eqref{yinface}.
\end{prop}

\section{Discussion of limit conic type metric}\label{limsection}
\subsection{Equation for the limit metric}
We first fix the reference metric to be the Fubini-Study metric.
\[
\omega=\partial\bar{\partial}\tilde{u}_0=\partial\bar{\partial}\log(\sum_\alpha|s_\alpha|^2)
\]
We want to see what's the limit of $\omega_t$ as $t\rightarrow R(X)$, where
\[
\omega_t=\omega+\partial\bar{\partial}\phi
\]
is solution of continuity equation \ref{CMAt}.
Equivalently, under the toric coordinate,
\[
\omega_t=\frac{\partial^2 u}{\partial \log z_i\partial\log z_j}d\log z_i\wedge d\log z_j=-\sqrt{-1} u_{ij}dx_id\theta_j
\]
 where $u=u_t$ is the solution of real Monge-Amp\`{e}re equation \ref{RMAt}.

Let $\sigma=\sigma_t$ be the holomorphic transformation given by
\[
\sigma_t(x)=x+x_t
\]
Assume $x_t=(x_t^1,\cdots, x_t^n)$, then under complex coordinate, we have
\[
\sigma_t(\{z_i\})=\{e^{x_t^i/2}z_i\}
\]
By the analysis of previous section, we do the following transformation.
\begin{equation}\label{transform}
U(x)=\sigma_t^*u(x)-u(x_t)=u(x+x_t)-u(x_t),\quad \tilde{U}_t(x)=\sigma_t^*\tilde{u}_0(x)-\tilde{u}_0(x_t)=\tilde{u}_0(x+x_t)-\tilde{u}_0(x_t)
\end{equation}
Note that $w_t(x)=tu+(1-t)\tilde{u}_0$. Then $U=U_t(x)$ satisfy the following Monge-Amp\`{e}re equation
\begin{align}
\det(U_{ij})=e^{-t
U-(1-t)\tilde{U}-w(x_t)}\tag*{$(**)'_t$}\label{TRMAt}
\end{align}

By Proposition \ref{main}, we know that $Q=-\frac{R(X_\triangle)}{1-R(X_\triangle)}P_c$ lies on the boundary of $\triangle$. Let $\mathcal{F}$ be the minimal face of $\triangle$ which contains $Q$. Now the observation is
\begin{prop}\label{limref}
There is a subsequence $t_i\rightarrow R(X)$, $\tilde{U}_{t_i}$
converge locally uniformly to a convex function of the form:
\begin{equation}\label{Uinf}
\tilde{U}_\infty=\log\left(\sum_{p_\alpha\in\mathcal{F}} b_\alpha e^{\langle p_\alpha, x\rangle}\right)
\end{equation}
where $0<b_\alpha\le 1$ are some constants. For simplicity, we will use $\sum_\alpha{}'=\sum_{p_\alpha\in\mathcal{F}}$ to denote the sume over all the \textbf{vertex lattice points} contained in $\mathcal{F}$.
\end{prop}
\begin{proof}
By \eqref{u0} and \eqref{transform}, we have
\begin{equation}\label{tUx}
\tilde{U}(x)=\log(\sum_\alpha e^{\langle p_\alpha, x+x_t\rangle})-\log(\sum_\alpha e^{\langle p_\alpha, x_t})=\log(\sum b(p_\alpha, t) e^{\langle p_\alpha, x\rangle})
\end{equation}
where
\[
b(p_\alpha, t)=\frac{e^{\langle p_\alpha, x_t\rangle}}{\sum_\beta e^{\langle p_\beta, x_t\rangle}}
\]
Since $0<b(p_\alpha, t)< 1$, we can assume there is a subsequence $t_i\rightarrow R(X)$, such that for any vertex lattice point $p_\alpha$,
\begin{equation}\label{limbt}
\lim_{t\rightarrow R(X)} b(p_\alpha, t)=b_\alpha
\end{equation}
We need to prove $b_\alpha\neq 0$ if and only if $p_\alpha\in\mathcal{F}$.

To prove this, we first note that
\begin{equation}\label{du0bt}
D\tilde{u}_0(x_t)=\frac{\sum_\alpha p_\alpha e^{\langle p_\alpha, x_t\rangle}}{\sum_\beta e^{\langle p_\beta, x_t\rangle}}=\sum_\alpha b(p_\alpha, t)p_\alpha
\end{equation}

By Lemma \ref{observe1}, $D\tilde{u}_0(x_t)\rightarrow y_\infty$.   So by letting $t\rightarrow R(X)$ in \eqref{du0bt} and using \eqref{limbt}, we get
\[
y_\infty=\sum_\alpha b_\alpha p_\alpha
\]

By Proposition \ref{main}, $y_\infty\in \partial\triangle$ lies on the same faces as $Q$ does, i.e. $\mathcal{F}$ is also the minimal face containing $y_\infty$, so we must have $b_\alpha=0$ if $p_\alpha\notin\mathcal{F}$. We only need to show if $p_\alpha \in\mathcal{F}$, then $b_\alpha\neq 0$.

If dim $\mathcal{F}$=k, then there exists k+1 vertex lattice points $\{p_1, \cdots, p_{k+1}\}$ of $\mathcal{F}$, such that the corresponding coefficient $b_i\neq 0$, $i=1,\cdots, k+1$, i.e. $\lim_{t\rightarrow R(X)}b(p_i, t)=b_i>0$.
\begin{rem}\label{vertex2}
 Here is why we need to assume $p_\alpha$ are all vertex lattice points.
\end{rem}

Let $p$ be any vertex point of $\mathcal{F}$, then
\[
p=\sum_{i=1}^{k+1} c_i p_i, \quad \mbox{where}\quad \sum_{i=1}^{k+1} c_i=1
\]
Then
\[
b(p,t)=\frac{e^{\langle\sum_{i=1}^{k+1}c_i p_i,
x_t\rangle}}{\sum_\beta e^{\langle p_\beta,
x_t\rangle}}=\prod_{i=1}^{k+1}\left(\frac{e^{\langle p_i,
x_t\rangle}}{\sum_\beta e^{\langle p_\beta, x_t\rangle}}
\right)^{c_i}=\prod_{i=1}^{k+1}b(p_i, t)^{c_i}\stackrel{t\rightarrow
R(X)}{-\!\!-\!\!\!\longrightarrow} \prod_{i=1}^{k+1}b_i^{c_i}>0
\]
\end{proof}
We can state a real version of Theorem \ref{thm2}
\begin{thm}\label{thm3}
There is a subsequence $t_i\rightarrow R(X)$, $U_{t_i}(x)$ converge
to a smooth entire solution of the following equation on
$\mathbb{R}^n$
\begin{align}
\det(U_{ij})=e^{-R(X)U(x)-(1-R(X))\tilde{U}_\infty(x)-c}\tag*{$(**)'_\infty$}\label{limiteq}
\end{align}
$c=\lim_{t_i\rightarrow R(X)}w(x_{t_i})$ is some constant.
\end{thm}

\subsection{Change to Complex Monge-Amp\`{e}re equation}
The proof of Theorem \ref{thm3} might be done by theory of real
Monge-Amp\`{e}re equation. But here, we will change our view and
rewrite \ref{TRMAt} as a family of complex Monge-Amp\`{e}re
equations. This will alow us to apply some standard estimates in the
theory of complex Monge-Amp\`{e}re equations.

We rewrite the formula for $\tilde{U}(x)$ \eqref{tUx}  as
\begin{equation}\label{transdif}
e^{\tilde{U}}=\frac{\sum_\alpha b(p_\alpha,t) e^{\langle p_\alpha, x\rangle}}{\sum_\beta e^{\langle p_\beta, x\rangle}}\sum_\beta e^{\langle p_\beta,x\rangle}=\frac{\sum_\alpha b(p_\alpha,t)|s_\alpha|^2}{\sum_\beta |s_\beta|^2}e^{\tilde{u}_0}
=(\sum_\alpha b(p_\alpha,t) \|s_\alpha\|^2) e^{\tilde{u}_0}
\end{equation}
$s_\alpha$ is the holomorphic section of $K_X^{-1}$ corresponding to lattice point $p_\alpha$. Here and in the following $\|\cdot\|=\|\cdot\|_{FS}$ is the Fubini-Study metric on $K_X^{-1}$.

\ref{TRMAt} can then be rewritten as
\[
\det(U_{ij})=e^{-t\psi}e^{-\tilde{u}_0}\left(\sum_\alpha b(p_\alpha,t)\|s_\alpha\|^2\right)^{-(1-t)}e^{-w(x_t)}
\]
By \eqref{toricvol} and \eqref{torich},
\ref{TRMAt} can finally be written as the complex Monge-Amp\`{e}re
equation
\begin{align}
 (\omega+\partial\bar{\partial}\psi)^n= e^{-t\psi}\left(\sum_\alpha b(p_\alpha,t)\|s_\alpha\|^2\right)^{-(1-t)}e^{h_\omega-w(x_t)}\omega^n\tag*{$(***)_t$}\label{TCMAt}
\end{align}
where
\begin{equation}\label{psit}
\psi=\psi_t=U-\tilde{u}_0
\end{equation}

Similarly for $\tilde{U}_\infty$ \eqref{Uinf}, we write
\[
e^{\tilde{U}_\infty}=\frac{\sum_\alpha{}'b_\alpha e^{\langle p_\alpha, x\rangle}}{\sum_\beta e^{\langle p_\beta, x\rangle}}\sum_\beta e^{\langle p_\beta,x\rangle}
=(\sum_\alpha{}'b_\alpha \|s_\alpha\|^2) e^{\tilde{u}_0}
\]
And the limit equation \ref{limiteq} becomes:
\begin{align}
 (\omega+\partial\bar{\partial}\psi)^n=e^{-R(X)\psi} \left(\sum_\alpha{}'b_\alpha\|s_\alpha\|^2\right)^{-(1-R(X))}e^{h_\omega-c}\omega^n\tag*{$(***)_\infty$}\label{cxlim1}
\end{align}


So we reformulate Theorem \ref{thm3} as the main Theorem \ref{thm2} in the introduction.


\subsection{Discussion on the conic behavior of limit metric}\label{conicint}
For any lattice point $p_\alpha\in\triangle$, let $D_{p_\alpha}=\{s_\alpha=0\}$ be the zero divisor of the corresponding holomorphic section $s_\alpha$.  By toric geometry , we have \[
D_{p_\alpha}=\{s_\alpha=0\}=\sum_{i=1}^{K} (\langle p_\alpha,v_i\rangle+1) D_i
\]
Here $v_i$ is the primitive inward normal vector to the i-th codimension one face, and $D_i$ is the toric divisor corresponding to this face.

Recall that  $\mathcal{F}$ is the minimal face containing $Q$. Let
$\{p_k^{\mathcal{F}}\}$ be all the vertex lattice points of
$\mathcal{F}$. They correspond to a sublinear system $\mathfrak{L}_\mathcal{F}$ of $|K_X^{-1}|$. The base locus of $\mathfrak{L}_\mathcal{F}$ is given by the schematic intersection
\[
Bs(\mathfrak{L}_\mathcal{F})=\bigcap_{k}D_{p_k^{\mathcal{F}}}
\]
The fixed components in $Bs(\mathfrak{L}_\mathcal{F})$ are
\begin{equation}\label{bslcs}
D^{\mathcal{F}}=\sum_{i=1}^r a_iD_i
\end{equation}
where
\[
\mathbb{N}\ni a_i=1+\min_{k}\langle p_k^{\mathcal{F}}, v_i\rangle >0, i=1, \dots, r
\]
For $i=1,\dots, K$, we always have $a_i=1+\min_{k}\langle p_k^{\mathcal{F}}, v_i\rangle\ge 0$. In \eqref{bslcs}, the coefficients $a_i$ are those with $a_i\ne 0$. 

%



Pick any generic point $p$ on $D^{\mathcal{F}}$.  $p$ lies on only one component of $D^{\mathcal{F}}$.   Without loss of generality, assume $p\in D_1$, and in a neighborhood $\mathcal{U}_p$ of $p$,  choose local coordinate $\{z_i\}$ such that $D_1$ is defined by $z_1=0$, then  the singular Monge-Amp\`{e}re equation \eqref{singMA1} locally becomes:

\begin{equation}\label{locsing1}
(\omega+\partial\bar{\partial}\psi)^n=|z_1|^{-2a_1(1-R(X))}f
\end{equation}
with $f$ a nonvanishing smooth function in $\mathcal{U}_p$. 

where $\{z_1=0\}$ is the current of integration along divisor $\{z_1=0\}$.

Note that we have the following singular conic metric in $\mathcal{U}_p$ 
\[
\eta=\frac{dz_1\wedge d\bar{z}_1}{|z_1|^{2\alpha}}+\sum_{i=2}^n dz_i\wedge d\bar{z}_i
\]
$\eta$ has conic singularity along $\{z_1=0\}$ with conic angle $2\pi(1-\alpha)$, and satisfies
\[
\eta^n=\frac{dz_1\wedge d\bar{z}_1\wedge\cdots\wedge dz_n\wedge d\bar{z}_n}{|z_1|^{2\alpha}} 
\]

Comparing this with \eqref{locsing1}, we expect that the limit K\"{a}hler metric around $p$ has conic singularity along $D_1$ with conic angle equal to  $2\pi(1-(1-R(X))a_1)$ and the same hold for generic points on $D_i$.  

\begin{rem}\label{overlook}
At present, it seems difficult to speculate the behavior of limit metric around higher codimensional strata of $D^{\mathcal{F}}$. See the discussion in example \ref{example2}.  We hope to return to this issue in future. In the first version of this paper, the author overlooked the higher codimensional strata of the singularity locus and gave a wrong statement of the main theorem \ref{thm2}. Professor Jian Song pointed out this to him. 
\end{rem}
\section{Proof of Theorem \ref{thm2}}

We are now in the general setting of complex Monge-Amp\`{e}re equations. \ref{cxlim1} is a complex Monge-Amp\`{e}re equation with poles at righthand side. \ref{TCMAt} can be seen as regularizations of \ref{cxlim1}. We ask if the solutions of \ref{TCMAt} converge to a solution of \ref{cxlim1}. Starting from Yau's work \cite{Yau}, similar problems has been considered by many people. Due to the large progress made by Ko\l{}odziej\cite{Kol}, complex Monge-Amp\`{e}re equation can be solved with very general, usually singular, righthand side. Ko\l{}odziej's result was also proved by first regularizing the singular Monge-Amp\`{e}re equation, and then taking limit back to get solution of original equation.

We will derive several apriori estimate to prove Theorem \ref{thm2}.
For the $C^0$-estimate, the upper bound follows from how we
transform the potential function in \eqref{transform}. The lower
bound follows from a Harnack estimate for the transformed potential
function which we will prove using Tian's argument in \cite{T4}. For
the proof of partial $C^2$-estimate, higher order estimates and
convergence of solutions, we use some argument similar to that used
by Ruan-Zhang \cite{RuZh}, and Demailly and Pali \cite{DP}.

\subsection{\texorpdfstring{$C^0$}{C0}-estimate}
We first derive the $C^0$-estimate for $\psi=U-\tilde{u}_0$. Let $\bar{v}=\bar{v}(x)$ be a piecewise linear function defined to be
\[
\bar{v}(x)=\max_{p_\alpha} \langle p_\alpha, x\rangle
\]
Then $u_0$ is asymptotic to $\bar{v}$ and it's easy to see that $|\bar{v}-\tilde{u}_0|\le C$. So we only need to show that $|U(x)-\bar{v}(x)|\le C$. Here and in the following, $C$ is some constant independent of $t\in [0, R(X))$.

One side is easy. Since $DU(\mathbb{R}^n)=\triangle$ and $U(0)=0$, we have for any $x\in\mathbb{R}^n$,
$U(x)=U(x)-U(0)= DU(\xi)\cdot x\le \bar{v}(x)$.
$\xi$ is some point between 0 and $x$.  So
\[\psi=(U-\bar{v})+(\bar{v}-\tilde{u}_0)\le C\].


To prove the lower bound for $\psi$, we only need to prove a Harnack inequality
\begin{prop}
\begin{equation}\label{Harnack}
\sup_X(-\psi)\le n\sup_X \psi+C(n)t^{-1}
\end{equation}
\end{prop}
For this we use the same idea of proof in \cite{T4}. First we rewrite the \ref{TCMAt} as
\begin{equation}\label{c0TCMAt}
(\omega+\partial\bar{\partial}\psi)^n=e^{-t\psi+F-B_t}\omega^n
\end{equation}
where
\[
B_t=(1-t)\log\left(\sum_\alpha b(p_\alpha,t)\|s_\alpha\|^2\right), \quad F=h_\omega-w(x_t)
\]

Now consider a new continuous family of equations
\begin{align}
 (\omega+\partial\bar{\partial}\theta_s)^n= e^{-s\theta_s+F-B_t}\omega^n\tag*{$\eqref{c0TCMAt}_s$}\label{news}
\end{align}

Define $S=\{s'\in [0,t]|  \mbox{\ref{news} is solvable for } s\in
[s',t]\}$. We want to prove $S=[0,t]$. Since \eqref{c0TCMAt} has a
solution $\psi$, $t\in S$ and $S$ is nonempty. It is sufficient to
show that $S$ is both open and closed.

For openness, we first estimate the first eigenvalue of the metric $g_\theta$ associated with the K\"{a}hler form $\omega_\theta=\omega+\partial\bar{\partial}\theta$ for the solution $\theta$ of \ref{news}.

\begin{eqnarray}\label{riclower}
Ric(\omega_\theta)&=&s\partial\bar{\partial}\theta-\partial\bar{\partial}F+\partial\bar{\partial}B_t+Ric(\omega)\nonumber\\
&=&s\partial\bar{\partial}\theta+\omega+(1-t)(\sigma^*\omega-\omega)=s(\partial\bar{\partial}\theta+\omega)+(t-s)\omega+(1-t)\sigma^*\omega\nonumber\\
&=&s\omega_\theta+(t-s)\omega+(1-t)\sigma^*\omega
\end{eqnarray}
In particular, $Ric(\omega_\theta)>s\omega_\theta$.  So by Bochner's formula, the first nonzero eigenvalue $\lambda_1(g_{\theta_s})>s$. This gives the invertibility of linearization operator $(-\Delta_s)-s$ of equation \ref{news}, so the openness of solution set $S$ follows.

To prove closedness, we need to derive apriori estimate. First define the functional:
\[
I(\theta_s)=\frac{1}{V}\int_X \theta_s(\omega^n-\omega_{\theta_s}^n), \quad J(\theta_s)=\int_0^1 \frac{I(x\theta_s)}{x}dx
\]
Then we have
\begin{lem}\cite{BM, T4}\label{IJrel}
 \begin{enumerate}
  \item[(i)]
\begin{equation}\label{IJineq}
 (n+1)J(\theta_s)/n\le I(\theta_s)\le (n+1)J(\theta_s)
\end{equation}
 \item[(ii)]
\[
\frac{d}{ds}(I(\theta_s)-J(\theta_s))=-\frac{1}{V}\int_X\theta_s(\Delta_s\dot{\theta_s})\omega_{\theta_s}^n
\]

 \end{enumerate}
\end{lem}
Using $\lambda_1(g_{\theta_s})>s$,  Lemma \ref{IJrel}.(ii) gives
\begin{lem}\cite{BM,T4}\label{I-Jdec}
 $I(\theta_s)-J(\theta_s)$ is monotonically increasing.
\end{lem}
Let's recall Bando-Mabuchi's estimate for Green function.
\begin{prop}\cite{BM}
For every m-dimensional compact Riemannian manifold (X,g) with

$diam(X,g)^2 Ric(g)\ge -(m-1)\alpha^2$, there exists a positive constant $\gamma=\gamma(m,\alpha)$ such that
\begin{equation}\label{estgreen}
G_g(x,y)\ge -\gamma(m,\alpha)diam(X,g)^2/V_g
\end{equation}
Here the Green function $G_g(x,y)$ is normalized to satisfy
\[
\int_M G_g(x,y)dV_g(x)=0
\]
\end{prop}
Bando-Mabuchi used this estimate to prove the key estimate:
\begin{prop}\cite{BM}
Let
\[
\mathcal{H}^s=\{\theta\in C^\infty(X); \omega_\theta=\omega+\partial\bar{\partial}\theta>0, Ric(\omega_\theta)\ge s\omega_\theta\}
\]
then for any $\theta\in\mathcal{H}^s$, we have
\begin{enumerate}
 \item[(1)]
\begin{equation}\label{infphi}
\sup_X(-\theta)\le\frac{1}{V}\int_X(-\theta)\omega_\theta^n+C(n)s^{-1}
\end{equation}
\item[(2)]
\begin{equation}\label{oscest}
 Osc(\theta)\le I(\theta)+C(n)s^{-1}
\end{equation}

\end{enumerate}
\end{prop}

\begin{prop}\label{solve0}
 \ref{news} is solvable for $0\le s\le t$.
\end{prop}
\begin{proof}
From \ref{news}, there exists $x_s\in X$ such that
$
-s\theta_s(x_s)+F(x_s)-B_t(x_s)=0
$, so $|\theta_s(x_s)|=\frac{1}{s}|F-B_t|(x_s)\le C_ts^{-1}$.  By \eqref{oscest} and $I\le (n+1)(I-J)$ (by \eqref{IJineq}), we get 
\[
\sup_X\theta_s\le Osc(\theta)+\theta(x_s)\le (n+1) (I-J)(\theta)+C(n)s^{-1}+C_ts^{-1}
\]
By Lemma \ref{I-Jdec}, for any $\delta>0$, we get uniform estimate for $\sup_X\theta_s$ and hence also $\inf_X\theta_s$ for $s\in [\delta, t]$. So $\|\theta_s\|_{C^0}\le C\delta^{-1}$. We can use Yau's estimate to get $C^2$ and higher order estimate. So we can solve \ref{news} for $s\in [\delta, t]$, for any $\delta>0$. 

On the otherhand, by Yau's theorem, we can solve \ref{news} for $s=0$. And by implicit function theorem, we can solve
\ref{news} for $s\in [0, \tau)$ for $\tau$ sufficiently small. We can pick $\delta$ such that $\delta<\tau$, so we get solution of \ref{news} for $s\in [\delta, \tau)$ in two ways. They must coincide by the recent work of Berndtsson \cite{Bern} on the uniqueness of solutions for the twisted K\"{a}hler-Einstein equation \eqref{riclower}. So we complete the proof.
\end{proof}

Then one can use the same argument as in \cite{T4} to prove
\begin{prop}\cite{T4}
\begin{equation}\label{tianest}
-\frac{1}{V}\int_X\theta \omega_\theta^n\le \frac{n}{V}\int_X\theta\omega^n\le n\sup_X\theta
\end{equation}
\end{prop}
\begin{proof}
 First by taking derivatives to equation \ref{news}, we get
\[
\Delta_s\dot{\theta}=-\theta-s\dot{\theta}
\]
So
\begin{eqnarray*}
 \frac{d}{ds}(I-J)(\theta_s)&=&-\int_X\theta\frac{d}{ds}\omega_\theta^n=-\frac{d}{ds}\left(\int_X\theta\omega_\theta^n\right)+\int_X\dot{\theta}\omega_\theta^n\\
&=&-\frac{d}{ds}\left(\int_X\theta\omega_\theta^n\right)-\frac{1}{s}\int_X\theta\omega_\theta^n=-\frac{1}{s}\frac{d}{ds} \left(s\int_X\theta\omega_\theta^n\right)
\end{eqnarray*}
So
\[
\frac{d}{ds}(s(I-J)(\theta_s))-(I-J)(\theta_s)=-\frac{d}{ds}\left(s\int_X\theta\omega_\theta^n\right)
\]
Since $\theta_s$ can be solved for $s\in [0,t]$, and $\theta_t=\psi=\psi_t$, we can integrate to get
\[
t(I-J)(\psi)-\int_0^t (I-J)(\theta_s)ds=-t\int_X\psi\omega_\psi^n
\]
Divide both sides by $t$ to get
\[
(I-J)(\psi)-\frac{1}{t}\int_0^t(I-J)(\theta_s)ds=-\int_X\psi\omega_\psi^n
\]
By lemma \ref{IJrel}.(i), we can get
\[
\frac{n}{n+1}\int_X\psi(\omega^n-\omega_\psi^n)=\frac{n}{n+1}I(\psi)\ge -\int_X\psi\omega_\psi^n
\]
\eqref{tianest} follows from this inequality imediately.
\end{proof}

Combine \eqref{tianest} with Bando-Mabuchi's estimate \eqref{infphi} when $s=t$, we then prove the Harnack estimate \eqref{Harnack}. So we can derive the
lower bound of $\psi$ from the upper bound of $\psi$ and $C^0$-estimate is obtained.

\begin{rem}\label{gap}
Professor Jian Song showed me that one can modify Tian's argument to prove Harnack inequality so that in the Proposition \ref{solve0} one only needs to solve \ref{news} for $s\in (\delta, t]$ with $\delta$ sufficiently small and in this way we can avoid the use of Berndtsson's recent uniqueness result. 
\end{rem}

\subsection{Partial \texorpdfstring{$C^2$}{C2}-estimate}
\ref{CMAt} is equivalent to
\[
Ric(\omega_\phi)=t\omega_\phi+(1-t)\omega
\]
From our transformation \eqref{transform}, we get
\begin{equation}\label{ricpsi}
Ric(\omega_\psi)=t\omega_\psi+(1-t)\sigma^*\omega
\end{equation}
In particular, $Ric(\omega_\psi)>t\omega_\psi$. We will some argument similar to that used by Ruan-Zhang (see the proof of Lemma 5.2 in \cite{RuZh})

Let $f=tr_{\omega_\psi}\omega$ and $\Delta'$ be the complex Laplacian associated with K\"{a}hler metric $\omega_\psi$. As in \cite{Yau2}, we can calculate
\begin{eqnarray*}
\Delta' f=g'^{i\bar{l}}g'^{k\bar{j}}R'_{k\bar{l}}g_{i\bar{j}}+g'^{i\bar{j}}g'^{k\bar{l}}T^{\alpha}_{i,k}T^{\beta}_{\bar{j},\bar{l}}g_{\alpha\bar{\beta}}-g'^{i\bar{j}}g'^{k\bar{l}}
S_{i\bar{j}k\bar{l}}
\end{eqnarray*}
Here the tensor $T^{\alpha}_{i,j}=\tilde{\Gamma}^{\alpha}_{i j}-\Gamma^{\alpha}_{ij}$ is the difference of Levi-Civita connections $\tilde{\Gamma}$ and $\Gamma$ associated with $g_{\omega}$ and $g'=g_{\omega_\psi}$ respectively. $R'_{k\bar{j}}$ is the Ricci curvature of $\omega_\psi$ and $S_{i\bar{j}k\bar{l}}$ is the curvature of reference metric $\omega$. Let $\nabla'$ be the gradient operator associated with $g_{\omega_\psi}$, then
\begin{eqnarray}\label{schwarz}
\Delta'\log f&=&\frac{\Delta'f}{f}-\frac{|\nabla'f|_{\omega_\psi}^2}{f^2}\nonumber\\
&=&\frac{g'^{i\bar{l}}g'^{k\bar{j}}R'_{k\bar{l}}g_{i\bar{j}}}{f}-\frac{g'^{i\bar{j}}g'^{k\bar{l}}
S_{i\bar{j}k\bar{l}}}{f}+\frac{g'^{i\bar{j}}g'^{k\bar{l}}T^{\alpha}_{i,k}T^{\beta}_{\bar{j},\bar{l}}g_{\alpha\bar{\beta}}}{f}-\frac{g'^{p\bar{q}}g'^{i\bar{j}}g'^{k\bar{l}}
T^{\alpha}_{ip}T^{\bar{\beta}}_{\bar{l}\bar{q}}g_{\alpha\bar{j}}g_{k\bar{\beta}}}{f^2}\nonumber\\
&=&\frac{\sum_i\mu_i^{-2}R_{i\bar{i}}}{\sum_i\mu_i^{-1}}-\frac{\sum_{i,j}\mu_i^{-1}\mu_j^{-1}S_{i\bar{i}j\bar{j}}}{\sum_i\mu_i^{-1}}+
\frac{\sum_{i,k,\alpha}\mu_i^{-1}\mu_k^{-1}|T^\alpha_{ik}|^2}{\sum_i\mu_i^{-1}}-\frac{\sum_p \mu_p^{-1}|\sum_{i}\mu_i^{-1}T^{i}_{ip}|^2}{(\sum_i\mu_i^{-1})^2}\nonumber\\
&\ge& t-C \sum_i\mu_i^{-1}=t-C f
\end{eqnarray}
In the 3rd equality in \eqref{schwarz}, for any fixed point $P\in X$, we chose a coordinate near $P$ such that $g_{i\bar{j}}=\delta_{ij}$, $\partial_kg_{i\bar{j}}=0$. We can assume
$g'=g_{\omega_\psi}$ is also diagonalized so that
\[
g'_{i\bar{j}}=\mu_i\delta_{ij}, \quad\mbox{with } \mu_i=1+\psi_{i\bar{i}}
\]
For the last inequality in \eqref{schwarz}, we used $Ric(\omega_\psi)>t\omega_\psi$ and the inequality:
\begin{eqnarray*}
\sum_{p}\mu_p^{-1}|\sum_i\mu_i^{-1}T^{i}_{ip}|^2&=&\sum_p\mu_p^{-1}\left|\sum_i\mu_i^{-1/2}T^{i}_{ip}\mu_i^{-1/2}\right|^2\\
&\le& (\sum_{p,i}\mu_p^{-1}\mu_i^{-1}|T^{i}_{ip}|^2)(\sum_i\mu_i^{-1})\\
&\le& (\sum_{p,i,\alpha}\mu_p^{-1}\mu_i^{-1}|T^{\alpha}_{ip}|^2)(\sum_i\mu_i^{-1})
\end{eqnarray*}
So
\[
\Delta'(\log f-\lambda\psi)\ge t-Cf-\lambda tr_{\omega_\psi}(\omega_\psi-\omega)=(\lambda-C)f-(\lambda n-t)=C_1 f-C_2
\]
for some constants $C_1>0$, $C_2>0$, if we choose $\lambda$ to be sufficiently large. So at the maximum point $P$ of the function $\log f-\lambda\psi$,
we have
\[
0\ge\Delta'(\log f-\lambda\psi)(P)\ge C_1f(P)-C_2
\]
So
\[
f(P)=tr_{\omega_\psi}(\omega)(P)\le C_3
\]
So for any point $x\in X$, we have
\[
tr_{\omega_\psi}\omega(x)\le C_3 e^{\lambda(\psi(x)-\psi(P))}\le C_3e^{\lambda osc(\psi)}
\]
By $C^0$-estimate of $\psi$, we get the estimate
$
tr_{\omega_{\psi}}\omega\le C_4
$. So $\omega_\psi\ge C_4\omega$, i.e. $\mu_i\ge C_4$.

Now by \eqref{c0TCMAt},
\[
\prod_j\mu_j=\frac{\omega_\psi^n}{\omega^n}=e^{-t\psi+F-B}
\]
with $F=h-w(x_t)$ and $B=(1-t)\log\left(\sum_\alpha b(p_\alpha,t)\|s_\alpha\|^2\right)$.
So by the $C^0$-estimate of $\psi$, we get
\[
\mu_i=\frac{\prod_j\mu_j}{\prod_{j\neq i}\mu_j}\le\frac{e^{-t\psi+F-B}}{C_4^{n-1}}\le C_5 e^{-B}
\]
In conclusion, we get the partial $C^2$-estimate
\begin{equation}\label{ellip}
C_4\omega\le \omega_\psi\le C_5 e^{-B}\omega
\end{equation}
\begin{rem}
The partial $C^2$-upper bound $\omega_\psi\le C_5 e^{-B}\omega$ can also be proved by maximal principle. In fact, let
\begin{equation}\label{Lambda}
\Lambda=\log(n+\Delta\psi)-\lambda\psi+B
\end{equation}
where $\Delta=\Delta_\omega$ is the complex Laplacian with respect to reference metric $\omega$.
Then by standard calculation as in Yau \cite{Yau}, we have
\begin{eqnarray}\label{2ndcal}
\Delta'\Lambda&\ge&\left(\inf_{i\neq j}S_{i\bar{i}j\bar{j}}+\lambda\right)\sum_i\frac{1}{1+\psi_{i\bar{i}}}+\left(\Delta F-\Delta B-t\Delta\psi-n^2\inf_{i\neq j}S_{i\bar{i}j\bar{j}}\right)\frac{1}{n+\Delta\psi}-\lambda n+\Delta' B\nonumber\\
&=&\left(\inf_{i\neq j}S_{i\bar{i}j\bar{j}}+\lambda\right)\sum_i\frac{1}{1+\psi_{i\bar{i}}}+\left(\Delta F+nt-n^2\inf_{i\neq j}S_{i\bar{i}j\bar{j}}\right)\frac{1}{n+\Delta\psi}+\nonumber\\
&&+\sum_{i}B_{i\bar{i}}\left(\frac{1}{1+\psi_{i\bar{i}}}-\frac{1}{n+\Delta\psi}\right)-(\lambda n+t)
\end{eqnarray}

Since for each $i$,
$\frac{1}{n+\Delta\psi}\le\frac{1}{1+\psi_{i\bar{i}}}$, so $\frac{1}{n+\Delta\psi}\le \frac{1}{n}\sum_{i}\frac{1}{1+\psi_{i\bar{i}}}$. So
the second term on the right of \eqref{2ndcal} is bounded below by $-C_0\sum_i\frac{1}{1+\psi_{i\bar{i}}}$ for some positive constant $C_0>0$

For the 3rd term, we observe from \eqref{transform} and \eqref{transdif} that
\[
\partial\bar{\partial}B=(1-t)(\sigma^*\omega-\omega)\ge -(1-t)\omega
\]
So, since again $\frac{1}{n+\Delta\psi}\le\frac{1}{1+\psi_{i\bar{i}}}$, we have
\[
B_{i\bar{i}}\left(\frac{1}{1+\psi_{i\bar{i}}}-\frac{1}{n+\Delta\psi}\right)\ge -(1-t)\left(\frac{1}{1+\psi_{i\bar{i}}}-\frac{1}{n+\Delta\psi}\right)
\ge-(1-t)\frac{1}{1+\psi_{i\bar{i}}}
\]
By the above discussion, at the maximal point $P_t$ of $\Lambda$, we have
\begin{equation}\label{dellam}
0\ge\Delta'\Lambda\ge (\lambda+\inf_{i\neq j}S_{i\bar{i}j\bar{j}}-C_0-(1-t))\sum_{i}\frac{1}{1+\psi_{i\bar{i}}}-(\lambda n+t)=C_2\sum_i\frac{1}{1+\psi_{i\bar{i}}}-C_3
\end{equation}
for some constants $C_2>0$, $C_3>0$, by choosing $\lambda$ sufficiently large.

Now we use the following inequality from \cite{Yau}
\begin{eqnarray}\label{trick}
\sum_i\frac{1}{1+\psi_{i\bar{i}}}&\ge& \left(\frac{\sum_i(1+\psi_{i\bar{i}})}{\prod_j(1+\psi_{j\bar{j}})}\right)^{1/(n-1)}=(n+\Delta\psi)^{1/(n-1)}
e^{\frac{B-F+t\psi}{n-1}}\nonumber\\
&=&e^{\frac{\Lambda}{n-1}}e^{\frac{-F+(t+\lambda)\psi}{n-1}}
\end{eqnarray}
By \eqref{dellam} and \eqref{trick}, we get the bound
\[
e^{\Lambda(P_t)}\le C_4 e^{-(t+\lambda)\psi(P_t)}
\]
So we get estimate that for any $x\in X=X_\triangle$,
\[
(n+\Delta\psi)e^{-\lambda\psi}e^{B}\le e^{\Lambda(P_t)}\le C_4 e^{-(t+\lambda)\psi(P_t)}
\]
Since we have $C^0$-estimate for $\psi$, we get partial $C^2$-upper estimate:
\begin{equation}\label{partialc2}
(n+\Delta\psi)(x)\le C_4 e^{-(t+\lambda)\psi(P_t)}e^{\lambda\psi(x)}e^{-B}\le C_5 \left(\sum_\alpha b(p_\alpha,t)\|s_\alpha\|^2\right)^{-(1-t)}
\end{equation}
In particular,
\[
1+\psi_{i\bar{i}}\le C_5 e^{-B}
\]
which is same as $\omega_\psi\le C_5 e^{-B}$.
\end{rem}

\subsection{Higher order estimate and completion of the proof of Theorem \ref{thm2}}
For any compact set $K\subset X\backslash D$, we first get the gradient estimate by interpolation inequality:
\begin{equation}\label{gradest}
\max_{K}|\nabla\psi|\le C_K(\max_K \Delta\psi+\max_K|\psi|)
\end{equation}
Next, by the complex version of Evans-Krylov theory \cite{T5}, we have a uniform $C_{2,K}>0$, such that $\|\psi\|_{C^{2,\alpha}(K)}\le C_{2,K}$ sor some $\alpha\in(0,1)$. Now take derivative to the equation:
\[
\log\det(g_{i\bar{j}}+\psi_{i\bar{j}})=\log\det (g_{i\bar{j}})-t\psi+F-B
\]
to get
\begin{equation}\label{linearize}
g'^{i\bar{j}}\psi_{i\bar{j},k}=-t\psi_k+F_k-B_k+g^{i\bar{j}}g_{i\bar{j},k}-g'^{i\bar{j}}g_{i\bar{j},k}
\end{equation}
By \eqref{ellip}, \eqref{gradest} and $\|\psi\|_{C^{2,\alpha}(K)}\le C_{2,K}$,  \eqref{linearize} is a linear elliptic equation with $C^{\alpha}$ coefficients. By Schauder's estimate, we get $\|\psi_k\|_{C^{2,\alpha}}\le C$, i.e. $\|\psi\|_{C^{3,\alpha}}\le C$. Then we can iterate in \eqref{linearize} to get $\|\psi\|_{C^{r,\alpha}}\le C$ for any $r\in\mathbb{N}$. So we see that $(\psi=\psi(t))_{t< R(X)}\subset C^{\infty}(X\backslash D)$ is precompact in the smooth topology.

Now we can finish the proof of Theorem \ref{thm2} using argument from \cite{DP}
\begin{proof}[Proof of Theorem \ref{thm2}]
The uniform estimate $\|\psi\|_{L^\infty}$ implies the existence of a $L^1$-convergent sequence $(\psi_j=\psi_{t_j})_j$, $t_j\uparrow R(X)$ with limit $\psi_\infty\in \mathcal{PSH}(\omega)\cap L^\infty(X)$. We can assume that a.e.-convergence holds too. The precompactness of the family $(\psi_j)\subset C^\infty(X\backslash D)$ in the smooth topology implies the convergence of the limits over $X\backslash D$:
\begin{eqnarray*}
(\omega+\partial\bar{\partial}\psi_\infty)^n&=&\lim_{t_j\rightarrow R(X)}(\omega+\partial\bar{\partial}\psi_j)^n\\
&=&\lim_{t_j\rightarrow R(X)}e^{-t_j\psi_{t_j}}\left(\sum_\alpha b(p_\alpha,t_j)\|s_\alpha\|^2\right)^{-(1-t_j)}e^{h_\omega-w(x_{t_j})}\omega^n\\
&=&e^{-R(X)\psi_\infty} \left(\sum_\alpha{}'b_\alpha\|s_\alpha\|^2\right)^{-(1-R(X))}e^{h_\omega-c}\omega^n
\end{eqnarray*}

The fact that $\psi_\infty$ is a bounded potential implies that the global complex Monge-Amp\`{e}re measure $(\omega+\partial\bar{\partial}\psi_\infty)^n$ does not carry any mass on complex analytic sets. This follows from pluripotential theory (\cite{Klimek}) because complex analytic sets are pluripolar.  We conclude that $\psi_\infty$ is a global bounded solution of the complex Monge-Amp\`{e}re equation \ref{cxlim1} which belongs to the class
$\mathcal{PSH}(\omega)\cap L^{\infty}(X)\cap C^\infty(X\backslash D)$.
\end{proof}

\section{Example}\label{Example}
\begin{exmp}\label{example1}
$X_\triangle=Bl_p\mathbb{P}^n$. The polytope $\triangle$ is defined by
\[
x_i\ge -1, i=1,\cdots,n; \quad \sum_ix_i\ge -1; \quad\mbox{and}\quad -\sum_ix_i\ge-1
\]
Using the symmetry of the polytope, we can calculate that
\[
Vol(\triangle)=\frac{1}{n!}((n+1)^n-(n-1)^n)
\]
\[
P_c=\left(x_i=\frac{2(n-1)^n}{(n+1)((n+1)^n-(n-1)^n)}\right), \quad\mbox{and}\quad Q=\left(x_i=-\frac{1}{n}\right)
\]
So
\[
R(X_\triangle)=\frac{|\overline{OQ}|}{|\overline{P_cQ}|}=\left(1+\frac{|\overline{OP_c}|}{|\overline{OQ}|}\right)^{-1}=
\frac{(n+1)((n+1)^n-(n-1)^n)}{(n+1)^{(n+1)}+(n-1)^{(n+1)}}
\]
$\mathcal{F}$ is the (n-1)-dimensional simplex with vertices
\[
P_i=(-1,\cdots,\stackrel{i-th\; place}{n-2},\cdots,-1)\quad i=1,\cdots, n
\]
Let $e_j$ be the j-th coordinate unit vector, then $\langle P_i,e_j\rangle=-1$ for $i\neq j$. $\langle P_i, e_i\rangle=n-2$. $\langle P_i, \pm(1,\cdots,1)\rangle=\mp 1$. So $P_i$ corresponds to a holomorphic section $s_i$ with $\{s_i=0\}=(n-1)D_i+2D_\infty$, where $D_i$ is the toric divisor corresponding to the codimension one face with inward normal $e_i$. 

It's easy to see that $Bs(\mathfrak{L}_\mathcal{F})=2 D_\infty $. $D_\infty$ is the toric divisor corresponding to the simplex face with vertices $Q_i=(-1,\cdots,n,\cdots,-1)$.
If we view $X_\triangle$ as the projective compactification of $\mathcal{O}(-1)\rightarrow\mathbb{P}^{n-1}$, then $D_\infty$ is just the divisor added at infinity. So the limit metric should have conic singularity along $D_\infty$ with conic angle
\[
\theta=2\pi\times (1-(1-R(X))\times 2)=2\pi\frac{(n+1)^{n+1}-(3n+1)(n-1)^n}{(n+1)^{n+1}+(n-1)^{n+1}}
\]
In particular, if $n=2$, i.e. $X_\triangle=Bl_p\mathbb{P}^2$ which is the case of the figure in the Introduction, then
\[
R(X_\triangle)=\frac{6}{7},\quad \theta=2\pi\times\frac{5}{7}
\]
This agrees with the results of \cite{S} and \cite{ShZh}. In fact, the results in \cite{S} and \cite{ShZh} can be easily generalized to $Bl_p\mathbb{P}^n$ which give the same results as here.

\end{exmp}
\begin{exmp}\label{example2}
$X_\triangle=Bl_{p,q}\mathbb{P}^2$,
$P_c=\frac{2}{7}(-\frac{1}{3},-\frac{1}{3})$,
$-\frac{21}{4}P_c\in\partial\triangle$, so
$R(X_\triangle)=\frac{21}{25}$.

$\mathcal{F}=\overline{Q_1Q_2}$. $Q_1$ corresponds to holomorphic section $s_1$ with $\{s_1=0\}=2D_1+D_2$. $Q_2$ corresponds to $s_2$ with $\{s_2=0\}=D_1+2D_2$.  The fixed components in $Bs(\mathfrak{L}_\mathcal{F})$ are $D_1+D_2$. $D_1$ and $D_2$ are the divisors corresponding to the faces
$\overline{Q_4Q_3}$ and $\overline{Q_4Q_5}$ respectively. So at generic point of $D_1$ (or $D_2$), the conic angle along $D_1$ (or $D_2$) should be
\[
2\pi\times (1-(1-\frac{21}{25})\times 1)=2\pi\times \frac{21}{25}
\]
While around the point $p=D_1\cap D_2$, if we choose local coordinate around $p$ such that  $D_1=\{z_1=0\}$ and $D_2=\{z_2=0\}$, the ideal defining the base locus is $(z_1^2z_2, z_1z_2^2)=(z_1)(z_2)(z_1, z_2)$.  the limit singular Monge-Amp\`{e}re equation locally looks like
\[
(\omega+\partial\bar{\partial}\psi)^n=\frac{f}{|z_1|^{2\alpha}|z_2|^{2\alpha}(|z_1|^2+|z_2|^2)^{\alpha}}
\]
where $f$ is a nonvanishing smooth function near $p$ and $\alpha=1-R(X)=\frac{4}{21}$.  The author does not know a candidate singular K\"{a}hler metric as local model yet. See Remark \ref{overlook}. 
\end{exmp}
\vspace*{-20mm}
\begin{figure}[h]
\begin{center}
\setlength{\unitlength}{0.6mm}
\begin{picture}(60,60)
\put(-30,0){\line(1,0){60}} \put(0,-30){\line(0,1){60}}
\put(-21,-21){\line(0,1){42}} \put(-21,21){\line(1,0){21}}
\put(0,21){\line(1,-1){21}} \put(21,0){\line(0,-1){21}}
\put(-21,-21){\line(1,0){42}} \put(-2,-2){\line(1,1){12.5}}
\put(-2,-2){\circle*{2}} \put(10.5,10.5){\circle*{2}}
\put(-8,-8){$P_c$} \put(12,12){$-\frac{21}{4}P_c$}

\put(-21,21){\circle*{1}} \put(-21,-21){\circle*{1}} \put(21,-21){\circle*{1}}
\put(0,21){\circle*{1}}   \put(21,0){\circle*{1}}

\put(-23,23){$Q_5$} \put(-28,-23){$Q_4$} \put(23,-23){$Q_3$}
\put(1,23){$Q_1$} \put(22,2){$Q_2$}
\end{picture}
\end{center}
\end{figure}

Department of Mathematics, Princeton University, Princeton, NJ
08544, USA

E-mail address: chil@math.princeton.edu
\end{document}